\documentclass[12pt,a4paper]{amsart}
\usepackage{amssymb}

\usepackage{graphicx,amssymb,amsfonts,epsfig,amsthm,a4,amsmath,url}
\usepackage[latin1]{inputenc}

\usepackage{amsfonts,amssymb,verbatim}
\usepackage{latexsym}
\usepackage{mathrsfs}
\usepackage{stmaryrd}
\usepackage{xspace}
\usepackage{enumerate, paralist}
\usepackage[usenames,dvipsnames]{color}
 \usepackage{txfonts, pxfonts}

\newcommand{\Z}{\mathbb{Z}}
\newcommand{\HH}{\mathbb{H}}

\newcommand{\E}{\mathbb{E}}

\newcommand{\N}{\mathbb{N}}
\newcommand{\R}{\mathbb{R}}

\newcommand{\eps}{\varepsilon}

\begin{document}
\title[FPP on a hyperbolic graph]{First Passage percolation on a hyperbolic graph  admits bi-infinite geodesics}

\author{Itai Benjamini}
\address{The Weizmann Institute, Rehovot, Israel}

\email{itai.benjamini@gmail.com}

\author[Tessera]{Romain Tessera*} \address{Universit\'e Paris Cit\'e, Sorbonne Universit\'e, CNRS, IMJ-PRG, F-75013
Paris, France}\email{romtessera@gmail.com}
\thanks{{$*$} Supported in part by ANR project  ANR-24-CE40-3137}

\date{\today}

\baselineskip=16pt

\begin{abstract}
Given an infinite connected graph, a way to randomly perturb its metric is to assign random i.i.d.\ lengths to the edges. An open question attributed to Furstenberg is whether there exists a bi-infinite geodesic in first passage percolation on $\Z^2$, and more generally on $\Z^n$ for $n\geq 2$. Although the answer is generally conjectured to be negative, we give a positive answer for graphs satisfying some negative curvature assumption. Assuming only strict positivity and finite exponential moment for the random lengths, we prove that if a graph $X$ has bounded degree  and  contains a Morse geodesic (e.g.\ is non-elementary Gromov hyperbolic), then almost surely, there exists a bi-infinite geodesic in first passage percolation on $X$.  \end{abstract}

\maketitle





\newtheorem*{thmm}{Theorem}
\newtheorem*{propp}{Proposition}
\newtheorem*{corr}{Corollary}
\newtheorem*{Ack}{Acknowledgements}

\newtheorem{ThmIntro}{Theorem}
\newtheorem{CorIntro}[ThmIntro]{Corollary}
\newtheorem{PropIntro}[ThmIntro]{Proposition}

\newtheorem{thm}{Theorem}[section]
\newtheorem{cor}[thm]{Corollary}
\newtheorem{lem}[thm]{Lemma}
\newtheorem{clai}[thm]{Claim}
\newtheorem{prop}[thm]{Proposition}
\theoremstyle{definition}
\newtheorem{defn}[thm]{Definition}
\newtheorem{cla}{Claim}
\newtheorem{nota}[thm]{Notation}
\newtheorem{que}[thm]{Question}
\newtheorem{conj}[thm]{Conjecture}
\theoremstyle{remark}
\newtheorem{rem}[thm]{Remark}

\newtheorem{ex}[thm]{Example}
\newtheorem{exer}[thm]{Exercise}
\numberwithin{equation}{section}

\section{Introduction}

First passage percolation is a model of random perturbation of a given geometry.
In this paper, we shall restrict to the simplest model, where random i.i.d lengths are assigned to the edges of a fixed graph. We refer to  \cite{ADH,GK,Ke} for background and references.

Let us briefly recall how FPP is defined.
We consider a connected non-oriented graph $X$, whose set of vertices (resp.\ edges) is denoted by $V$ (resp.\ $E$).
For every function $\omega:E\to (0,\infty)$, we equip $V$ with the weighted graph metric $d_{\omega}$, where each edge $e$ has weight $\omega(e)$. In other words, for every $v_1,v_2\in V$, $d_{\omega}(v_1,v_2)$ is defined as the infimum over all path $\gamma=(e_1,\ldots, e_m)$ joining $v_1$ to $v_2$ of $|\gamma|_{\omega}:=\sum_{i=1}^m\omega(e_i)$. Observe that the graph metric on $V$ corresponds to the case where $\omega$ is constant equal to $1$, we shall simply denote it by $d$.
We will now consider a probability measure on the set of all weight functions $\omega$. We let $\nu$ be a probability measure supported on $[0,\infty)$. Our model consists in choosing independently at random the weights $\omega(e)$ according to $\nu$. More formally, we equip the space $\Omega=[0,\infty)^E$ with the product probability that we denote by $P$.

A famous open problem in percolation theory is whether with positive probability, first passage percolation on $\Z^2$ admits a bi-infinite geodesic.
In his Saint-Flour course from 84', Kesten attributes this question to Furstenberg (see \cite{Ke}).
Licea and Newman \cite{LN} made partial progress on this problem, which is still open, and mentioned that the conjecture that there are no such geodesics arose independently  in the physics community studying spin glass.
Wehr and   Woo \cite{WW} proved  absence of two sided infinite geodesic in a half plane, assuming the lengths distribution is continuously distributed with a finite mean.

For Riemannian manifolds, the existence  of bi-infinite geodesics  is influenced by  the curvature of the space.
It is well-known that complete simply connected non-positively curved Riemannian manifolds (such as the Euclidean space $\R^n$ or the hyperbolic space $\HH^n$) admit bi-infinite geodesics. Therefore, simply connected manifolds without two-sided geodesics must have positively curved regions. It is easy to come up with examples of complete Riemannian surfaces with bubble-like structures that create short cuts avoiding larger and larger balls  around some origin. See \cite{Ba} for background.

 To help the reader's intuition, let us roughly describe a similar example in the  graph setting.
Starting with the standard Cayley graph of $\Z^2$, it is not difficult to choose edges lengths among the two possible values $1/10$ and $1$, such that the resulting weighted graph has no bi-infinite geodesics. To do so, consider a sequence of squares $C_n$ centered at the origin, whose size grows faster than any exponential sequence (e.g.\ like $n^n$). Then attribute length $1/10$ to the edges along $C_n$ for all $n$, and $1$ to all other edges. This creates large ``bubbles"  with relatively small neck in the graph (which in a sense can be interpreted as  large positively curved regions). One easily checks for all $n$, for  every pair of points at large enough distance from the origin, any geodesic between them never enters $C_n$ (as it is more efficient to go around the shorter edges of its boundary, than traveling inside it). Random triangulations admit bubbles and indeed no two sided
infinite geodesics.  Geodesics go via ``mountain passes" these are random  analogous of the Morse geodesics defined below. See \cite{CL} for the study of FPP on random triangulations.

The Euclidean plane being flat, its discrete counterpart $\Z^2$ (and more generally $\Z^d$ for $d\geq 2$) is in some sense at criticality for the question of existence of bi-infinite geodesics in FPP. Therefore, one should expect that in presence of negative curvature,  FPP a.s.\ exhibits  bi-infinite geodesics. For instance, this should apply to FPP  on Cayley graphs of groups acting properly cocompactly by isometries on the hyperbolic space $\HH^d$. We will see that this is indeed the case.

Let us first introduce some notation.
Let $X$ be a simple graph, with no double edges. Recall that a path $\gamma=(e_1,\ldots, e_n)$ between two vertices $x,y$ is a sequence of consecutive edges joining $x$ to $y$. We denote  $(x=\gamma(0),\ldots,\gamma(n)=y)$ the set of vertices such that for all $0\leq i<n$, $\gamma(i)$ and $\gamma(i+1)$ are joined by the edge $e_{i+1}$. For all $i<j$, we shall also denote by $\gamma([i,j])$ the subpath $(e_{i+1},\ldots,e_j)$ joining $\gamma(i)$ to $\gamma(j)$.
Similarly, we define infinite paths indexed by $\N$ (resp.\ bi-infinite paths indexed by $\Z$).

\begin{defn}
Let $X$ be an infinite connected  graph, and let $C\geq 1$ and $K\geq 0$. A path $\gamma$ of length $n$ between two vertices $x$ and $y$ is called a $(C,K)$-quasi-geodesic finite path if for all $0<i<j\leq n$, $$j-i\left(=\left|\gamma[i,j]\right|\right)\leq Cd(\gamma(i),\gamma(j))+K.$$
Similarly, we define $(C,K)$-quasi-geodesic infinite (or bi-infinite) paths. An infinite (or a bi-infinite) path will simply be called a quasi-geodesic if it is $(C,K)$-quasi-geodesic for some constants $C$ and $K$.
\end{defn}
\begin{defn}
A bi-infinite path $\gamma$ in $X$ is called a {\em Morse quasi-geodesic} (resp.\ Morse geodesic) if it is a quasi-geodesic (resp.\ a geodesic) and if it satisfies the so-called Morse property: for all $C\geq 1$ and $K>0$, there exists $R$ such that every $(C,K)$-quasi-geodesic joining two points of $\gamma$ remains inside the $R$-neighborhood of $\gamma$.
\end{defn}
\begin{rem}\label{rem:quasi} It is well-known and easy to deduce from its definition that in a locally finite weighted graph, whose weights are bounded away from 0, a Morse quasi-geodesic always lies at
a bounded distance from a (Morse) bi-infinite geodesic. For the sake of completeness, let us outline the argument: let $\gamma$ be a Morse quasi-geodesic, and let $u_n$ and $v_n$ be two sequences of vertices on $\gamma$ escaping to infinity in opposite directions. For every $n$, let $\gamma_n$ be a geodesic segment joining $u_n$ to $v_n$. By the Morse property, $\gamma_n$ remains at uniformly bounded distance from $\gamma$. On the other hand, the condition on the degree and on the weights ensure that any ball of finite radius in $X$ is finite.   
Hence one can extract a sequence $\gamma_{i_n}$ that converges to some bi-infinite geodesic $\gamma_0$.  \end{rem}

\begin{thm}\label{thm:Main}
Let $X$ be an infinite connected graph with bounded degree, that contains a Morse quasi-geodesic $\gamma$. Assume $\E e^{c\omega_e}<\infty$ for some 
$c>0$ and $\nu(\{0\})=0$. Then for a.e.\ $\omega$, $X_{\omega}$ admits a bi-infinite geodesic. Moreover for every sequence of pairs of vertices $(x_n,y_n)$ going to infinity on opposite sides of $\gamma$, for a.e.\ $\omega$ there exists a finite  subset $A\subset X$ such that the $\omega$-geodesic between $x_n$ and $y_n$ crosses $A$.

\end{thm}

That is, if the underling graph admits a two sided infinite Morse geodesic in the graph metric, then  the random FPP metric will have a two sided infinite geodesic a.s.

Very recently Ahlberg and Hoffman \cite{AH}
made substantial progress on the structure of geodesic rays for FPP on $\Z^2$   solving the midpoint problem (from \cite{BKS}) which is related to the non existence of bi-infinite geodesic.
They showed that the probability a shortest path between $(-n,0)$ and $(n, 0)$ will go via $(0,0)$ is going to $0$ with $n$.  Their upper bound on the probability is going very slowly to $0$, and is  likely far from the truth. The theorem above shows that the probability of the midpoint event is not going to $0$ in the distance, along a Morse geodesic.
This can be useful in proving linear variance, see the conjecture at the end.

We briefly recall the definition of a hyperbolic graph. A geodesic triangle in a graph $X$ consists of a triplet of vertices
$x_0,x_1,x_2\in V$, and of geodesic paths $\gamma_0,\gamma_1,\gamma_2$ such that $\gamma_i$ joins $x_{i+1}$ to $x_{i+2}$ where $i\in \Z/3\Z$. Given $\delta\geq 0$, a geodesic triangle is called $\delta$-thin if for every $i\in \Z/3\Z$, every vertex $v_i$ on $\gamma_i$ lies at distance at most $\delta$ from either $\gamma_{i+1}$ or $\gamma_{i+2}$.
Said informally, a geodesic triangle is $\delta$-thin if every side is contained in the $\delta$-neighborhood of the other two sides. A graph is hyperbolic if there is $\delta < \infty$ so that all geodesic triangles are $\delta$-thin.
It is well-known \cite{Gr} that in a hyperbolic graph, any bi-infinite quasi-geodesic is Morse.
In particular, we deduce the following

\begin{cor}
Let $X$ be a hyperbolic graph with bounded degree containing at least one bi-infinite geodesic. Assume $\E e^{c\omega_e}<\infty$ for some 
$c>0$  and $\nu(\{0\})=0$. Then for a.e.\ $\omega$, $X_{\omega}$ admits a bi-infinite geodesic. Moreover for every sequence of pairs of vertices $(x_n,y_n)$ going to infinity on opposite sides of a fixed quasi-geodesic in the graph metric of $X$, for a.e.\ $\omega$ there exists a finite  subset $A\subset X$ such that the $\omega$-geodesic between $x_n$ and $y_n$ crosses $A$.
\end{cor}

Note that the case where $\nu$ is supported in an interval $[a,b]\subset (0,\infty)$ is essentially obvious. Indeed, for all $\omega$, the weighed graph $X_{\omega}$ is bi-Lipschitz equivalent to $X$.  We deduce that a Morse quasi-geodesic in $X$ remains a Morse quasi-geodesic in $X_{\omega}$ (adapting the definition to weighted graphs), and therefore lies at bounded distance from an actual bi-infinite geodesic.

We finish this introduction mentioning that Theorem \ref{thm:Main} applies to a wide class of Cayley graphs, including Cayley graphs of relatively hyperbolic groups, Mapping Class groups, and so on.
\section{Preliminary lemmas}
We start with a useful characterization of Morse quasi-geodesics, that one may take as a definition.
\begin{prop}\label{prop:DMS}\cite[Proposition 3.24 (3)]{DMS}
A bi-infinite quasi-geodesic $\gamma_0$ is Morse if and only if the following holds. For every $C\geq 1$, there exists $D\geq 0$  such that every path of length $\leq Cn$ connecting two
points $x,y$ on $\gamma$ at distance $\geq n$ crosses the $D$-neighborhood of the middle third of the segment of $\gamma_0$ joining $x$ to $y$.
\end{prop}
We deduce the following criterion, which we shall use in the sequel.
\begin{lem} \label{lem:MQG}
Let $X$ be an infinite connected graph with bounded degree. Assume $\gamma_0$ is a Morse quasi-geodesic. Then there exists an increasing function $\phi:\R_+\to \R_+$ such that $\lim_{t\to \infty}\phi(t)=\infty$, and with the following property.  Assume
\begin{itemize}
\item $x,y$ belong to $\gamma_0$;
\item $x'$ and $y'$ are vertices such that $d(x,x')=d(y,y')=R$, and $d(x',y')\geq 10R$;
\item $\gamma$ is  a path joining $x'$ to $y'$, and remains outside of the $R$-neighborhood of $\gamma_0$.
  \end{itemize}
  Then
 $$|\gamma|\geq \phi(R)d(x,y).$$
\end{lem}
\begin{proof}
Assume by contradiction that there exists a constant $C>0$, and for every $n$, an integer $R\geq n$, vertices $x,x',y,y'$ and a path $\gamma$ of length $\leq Cd(x,y)$ such that $d(x,x')=d(y,y')=R$, $d(x',y')\geq 10R$, and $\gamma$ avoids the $R$-neighborhood of $\gamma_0$. By choosing $n$ large enough, we can assume that $R>D$. Applying Proposition \ref{prop:DMS} to the path obtained by concatenating $\gamma$ with geodesics from $x'$ to $x$ and from $y'$ to $y$ yields the desired contradiction.
\end{proof}

Our assumption $\nu(\{0\})=0$ is used to prove the following two lemmas.

\begin{lem}
Let $X$ be an infinite connected graph with bounded degree and assume that  $\nu(\{0\})=0$.
There exists an increasing function $\alpha:(0,\infty)\to (0,1]$ such that $\lim_{t\to 0}\alpha(t)=0$, and such that for all finite path $\gamma$ and all $\eps>0$,
$$
P\left(|\gamma|_{\omega}\leq \eps|\gamma|\right)\leq \alpha(\eps)^{|\gamma|}.
$$
\end{lem}
\begin{proof}
The assumption implies that for all $\lambda>0$,  there exists $\delta>0$ such that $\nu([0,0+\delta])<\lambda$.
Let $\gamma$ be a path of length $n$. Assume that $|\gamma|_{\omega}\leq \eps |\gamma|$, and let $N$  be the number of edges of $\gamma$ with $\omega$-length $\geq \delta $.  It follows that
$$\delta N \leq \eps n,$$
so we deduce that $N\leq \eps n/\delta$.
This imposes that  at least $(1-\eps/\delta)n$ edges of $\gamma$ have $\omega$-length $\leq \delta$.
Recall that by Stirling's formula, given some $0<\alpha<1$, the number of ways to choose $\alpha n$ edges in a path of length n is  $$\sim \frac{n^n}{(\alpha n)^{\alpha n}((1-\alpha) n)^{(1-\alpha) n}}= (1/\alpha)^{\alpha n}(1/(1-\alpha)^{(1-\alpha)n}.$$
Thus the probability that $\gamma$  has $\omega$-length at most $ \eps n$ is  less than a universal constant times
$$\frac{\lambda^{(1-\eps/\delta)n}}{(\eps/\delta)^{(\eps/\delta)n}(1-\eps/\delta)^{(1-\eps/\delta)n}}=\left(\frac{\lambda^{1-\eps/\delta}}{(\eps/\delta)^{\eps/\delta}(1-\eps/\delta)^{1-\eps/\delta}}\right)^{n}.$$
Note that $$\lim_{\eps\to 0}\frac{\lambda^{1-\eps/\delta}}{(\eps/\delta)^{\eps/\delta}(1-\eps/\delta)^{1-\eps/\delta}}=\lambda.$$
In other words, we have proved that for all $\lambda>0$, there exists $\eps>0$ such that
$$P\left(|\gamma|_{\omega}\leq \eps|\gamma|\right)\leq (2\lambda)^{|\gamma|}$$
which is equivalent to the statement of the lemma.
\end{proof}
\begin{lem}\label{lem:upperboundpath}
Let $X$ be an infinite connected graph with bounded degree, and let $o$ be some vertex of $X$. Assume $\nu(\{0\})=0$. Then there exists $c>0$ such that for a.e.\ $\omega$, there exists $r_1=r_1(\omega)$ such that for all finite path $\gamma$ such that\footnote{Here $d(\gamma,o)$ denotes the distance between $o$ and the set of vertices $\{\gamma(0),\gamma(1),\ldots \}$.} $d(\gamma,o)\leq |\gamma|$, one has
$$|\gamma|_{\omega}\geq c|\gamma|-r_1.$$
\end{lem}
\begin{proof} Let  $q$ be an upper bound on the degree of $X$, and let $n\geq 1$ be some integer. Every path of length $n$ lying at distance at most $n$ from $o$ is such that $d(o,\gamma(0))\leq 2n$, hence such a path is determined by a vertex in the ball $B(o,2n)$, whose size is at most $q^{2n}+1$, and a path of length $n$ originated from this vertex. Therefore
the number of such paths is at most $(q+1)^{3n}$.

On the other hand, we deduce from the previous lemma that for $c>0$ small enough, the probability that there exists some path $\gamma$ of length $n$, and at distance at most $n$ from $o$, and satisfying $|\gamma|_{\omega}\leq c|\gamma|$ is less than $1/(q+1)^{4n}$.
Hence the lemma follows by Borel-Cantelli.
\end{proof}

So far we have not used the exponential moment condition. The proof of Lemma \ref{lem:Tala} below relies on it, via a celebrated concentration inequality due to Talagrand, which we first recall.

\begin{thmm}\cite[Proposition 8.3]{Talagrand}).
Suppose that  $\E \exp(c\omega(e))<\infty$. Then there exists $C_1$ and $C_2$ such that for every pair of vertices $x,y$, and for every $u\geq 0$,
\begin{equation}\label{eq:expmoment}
P\left(|d_{\omega}(x,y)-\E d(x,y)|\geq u\right)\leq C_1\exp\left(-C_2\min\left\{\frac{u^2}{d(x,y)},u\right\}\right).
\end{equation}
\end{thmm}

\begin{lem}\label{lem:Tala}
 There exists $K>0$ such that for a.e.\ $\omega$, there exists $r_2$ such that for every vertex $y$, 
\[d_\omega(o,y)\leq  Kd(o,y) + r_2.\]
\end{lem}
\begin{proof}
We first fix $r\geq 1$ and some constant $K$. We start bounding the probability that there exists a vertex $B(o,Kr)$ which is at $\omega$-distance more than $K r$ from $o$. We will then pick $K$ large enough for our purpose. 

By triangular inequality, $\E d_\omega(x,y)\leq bd(x,y)$, where $b=\E \omega_e$. We deduce from Talagrand's inequality that
\begin{equation}\label{eq:expmomentSecond}
P\left(d_{\omega}(x,y)\geq bd(x,y)+u\right)\leq C_1\exp\left(-C_2\min\left\{\frac{u^2}{d(x,y)},u\right\}\right).
\end{equation}
Assume that $K-b\geq 1$ and apply (\ref{eq:expmomentSecond}) with $u=(K-b) r$, $x=o$, and $y$ a vertex in $B(o,r)$. This yields
\begin{equation}\label{eq:expmomentThird}
P\left(d_{\omega}(o,y)\geq Kr\right)\leq C_1\exp\left(-C_2(K-b)r\right\}.
\end{equation}
Indeed,  the condition that $d_{\omega}(o,y)\geq Kr$ can be written as $d_{\omega}(o,y)\geq br+ (K-b)r$, which implies $d_{\omega}(o,y)\geq bd(o,y)+ (K-b)r$. Besides, $u^2/d(o,y)\geq u^2/r\geq u$ since $(K-b)^2\geq K-b$.

 Let  $q$ be an upper bound on the degree of $X$, and choose $K$ such that $\delta:=C_2(K-b)-\log q >0$. By union bound (considering all points in $B(o,r))$), we get that the probability that there exists a vertex $y \in B(o,r)$ such that $d_\omega(o,y)\geq Kr$ is at most 
$C_1\exp\left(-C_2(K-b)r\right) q^r\leq C_1\exp\left(-\delta r\right)$. This is summable over $r$, so the conclusion follows by  applying Borel Cantelli lemma. \end{proof}

\section{Proof of Theorem \ref{thm:Main}}
Let $\gamma_0$ be some Morse quasi-geodesic. By Remark \ref{rem:quasi}, we do not loose generality by assuming that our Morse quasi-geodesic  $\gamma_0$ is a bi-infinite geodesic of the graph $X$. We let $o= \gamma_0(0)$ be some vertex.
We consider two sequences of vertices $(x_n)$ and $(y_n)$ on $\gamma_0$ escaping to infinity in opposite directions.

We let $\Omega'\subset \Omega$ be a measurable subset of full measure such that the conclusions of Lemmas \ref{lem:upperboundpath} and  \ref{lem:Tala}  hold.
For all $n$ and for all $\omega$, we pick measurably an $\omega$-geodesic $\gamma_{\omega}^n$ between $x_n$ and $y_n$. Note that Lemmas  \ref{lem:upperboundpath} and \ref{lem:Tala} imply that such a geodesic exists: by Lemma \ref{lem:Tala}, we have that $d_{\omega}(x_n,y_n)$ is finite, and by Lemma \ref{lem:upperboundpath}, paths of length $\geq M$ have $\omega$-length going to infinity as $M\to \infty$.

If we can prove the second part of Theorem \ref{thm:Main}, i.e.\ that for all $\omega\in \Omega'$, there exists a constant $R_{\omega}>0$ such that for all $n$, $d(\gamma_{\omega}^n,o)\leq R_{\omega}$,
then the first part of Theorem \ref{thm:Main} follows by a straightforward compactness argument\footnote{Indeed, observe that Lemma \ref{lem:upperboundpath} implies that $X_{\omega}$ is a.e.\ locally finite in the sense that $d_{\omega}$-bounded subsets of vertices are finite.}.
So we shall assume by contradiction that for some $\omega\in \Omega'$, there exists a sequence $R_n$ going to infinity such that $\gamma_{\omega}^n$ avoids $B(o,100R_n)$.

\begin{lem}
Assuming the above, there exist integers $q<p$ such that
$$d(\gamma_{\omega}^n(p),\gamma_0)=d(\gamma_{\omega}^n(q),\gamma_0)=R_n,$$
and such that for all $p\leq k\leq q$,
$$d(\gamma_{\omega}^n(k),\gamma_0)\geq R_n,$$
and
$$d(\gamma_{\omega}^n(p),\gamma_{\omega}^n(q))\geq 10R_n.$$
\end{lem}
\begin{proof}
(Note that this is obvious from a picture).
Let $\gamma_0(i)$ and $\gamma_0(j)$ with $i<0<j$ be the two points at distance $100R_n$ from $o=\gamma_0(0)$. Since  $\gamma_0$ is a geodesic, $\gamma_0((\infty,i])$ and $\gamma_0([j,\infty))$ are distance $200R_n$ from one another.  Let $r$ be the first time integer such that $d(\gamma_{\omega}^n(r), \gamma_0([j,\infty)))=100R_n$. By triangular inequality, $d(\gamma_{\omega}^n(r), \gamma_0((\infty,i]))\geq100R_n$, and since we also have $d(\gamma_{\omega}^n(r), o)\geq 100R_n$, we deduce that $$d(\gamma_{\omega}^n(r), \gamma_0)\geq 50R_n.$$

We let $p$ and $q$ be respectively the largest integer $\leq r$ and the smallest integer $\geq r$ such that
$$d(\gamma_{\omega}^n(p),\gamma_0)=d(\gamma_{\omega}^n(q),\gamma_0)=R_n.$$
Clearly, for all $p\leq k\leq q$,
$$d(\gamma_{\omega}^n(k),\gamma_0)\geq R_n.$$
Moreover, recall that $\gamma_{\omega}^n$ avoids $B(o,100R_n)$. Hence if $x$ and $y$ are points on $\gamma_0$ such that $d(\gamma_{\omega}^n(p),x)=d(\gamma_{\omega}^n(q),y)=R_n$, we deduce by triangular inequality that
$d(x,o)\geq 99R_n$ and $d(y,o)\geq 99R_n$. But since $x$ and $y$ lie on both sides of $o$, this implies that $d(x,y)\geq 198R_n$ (because $\gamma_0$ is a geodesic). Now by triangular inequality, we conclude that
$$d(\gamma_{\omega}^n(p),\gamma_{\omega}^n(q))\geq d(x,y)-2R_n \geq 196R_n\geq 10R_n.$$
So the lemma follows.
\end{proof}

\

\noindent{\bf End of the proof of Theorem \ref{thm:Main}}

We now let $i$ and $j$ be integers such that
$$d(\gamma_{\omega}^n(p),\gamma_0(i))=d(\gamma_{\omega}^n(q),\gamma_0(j))=R_n.$$
Note that by triangular inequality, $j-i=|\gamma_0([i,j])|\geq R_n$.

By Lemmas \ref{lem:upperboundpath} and  \ref{lem:MQG}, we have
\begin{eqnarray*}
|\gamma_{\omega}^n([p,q])|_{\omega} & \geq & c|\gamma_{\omega}^n([p,q])|-r_1\\
                                                                 & \geq & c\phi(R_n)|\gamma_0([i,j])|-r_1\\
                                                                 & = & c\phi(R_n)(j-i)-r_1.
\end{eqnarray*}
Note that $\gamma_{\omega}^n(p)$ and  $\gamma_{\omega}^n(q)$ both belong to $B(o,j-i+R_n)$. Hence by Lemma \ref{lem:Tala}, their $\omega$-distance to $o$ is at most $K(j-i+R_n)+r_2$.
Besides, since $\gamma_{\omega}^n$ is an $\omega$-geodesic between $x_n$ and $y_n$, we have
\begin{eqnarray*}
|\gamma_{\omega}^n([p,q])|_{\omega} & \leq & d_\omega(\gamma_{\omega}^n(p),o)+d_\omega(o,\gamma_{\omega}^n(q))\\
                                                                 & \leq & 2K(j-i+R_n)+2r_2.
 \end{eqnarray*}
Gathering these inequalities, we obtain 
\[ c\phi(R_n)(j-i)-r_1\leq 2K(j-i+R_n)+2r_2.\]
Gathering the terms proportional to $j-i$, we get
\[ (c\phi(R_n)-2K)(j-i)\leq 2KR_n+2r_2+r_1.\]
Now using that $j-i\geq R_n$, this gives
\[ (c\phi(R_n)-2K)R_n\leq 2KR_n+2r_2+r_1,\]
which yields a contradiction since $\phi(R_n)\to \infty$ as $n\to \infty$.
This ends the proof of Theorem \ref{thm:Main}.


\section{Remarks and questions}

\begin{itemize}

\item (Cayley graphs) We do not know a single example of an infinite Cayley graph, for which FPP a.s.\ admits no bi-infinite geodesics  (to fix the ideas, assume the edge length distribution is supported on the interval $[1,2]$).
\item (Adding dependence)
Given a hyperbolic Cayley graph, rather than considering independent edges lengths it is natural to consider other group invariant distributions. Under which natural conditions (mixing?) on this distribution do bi-infinite geodesics  a.s. exists?

 \item (Poisson Voronoi and other random models)
A variant of random metric perturbation is obtained via Poisson Voronoi tiling of a measure metric space. It seems likely that our method of proof  applies to the hyperbolic Poisson Voronoi tiling, see \cite{BPP}.  Recently other versions of  random hyperbolic triangulations were constructed,  \cite{AR} \cite{C}. Since those are not obtained by perturbing an underling  hyperbolic space, our proof does not apply to this setting.

\item(Variance along Morse geodesics)
We  {\em conjecture} that under a suitable moment condition on the edge-length distribution, the variance of the random distance grows linearly along the  Morse quasi-geodesic, unlike in Euclidean lattices \cite{BKS}. For lengths which are bounded away from zero and infinity, Morse's property ensures that geodesics remain at uniformly bounded distance from $\gamma$, hence reducing the problem to "filiform graphs", i.e.\  graphs quasi-isometric to $\Z$. A (very) special class of filiform graphs is dealt with in \cite{A} where a linear variance growth is proven.

\end{itemize}

\bigskip
\footnotesize

\end{document}